\documentclass[11pt]{elsarticle}

\usepackage{lineno,hyperref}
\usepackage[arrow, matrix]{xy}
\usepackage{amsmath}
\usepackage{cancel}
\usepackage{amssymb}
\usepackage{amsthm}
\usepackage[mathscr]{eucal}
\input{epsf}
\theoremstyle{definition}
\newtheorem{definition}{Definition}
\newtheorem{theorem}[definition]{Theorem}
\newtheorem{proposition}[definition]{Proposition}
\newtheorem{lemma}[definition]{Lemma}

\theoremstyle{remark}
\newtheorem{remark}[definition]{Remark}

\newcounter{enumctr}

\newcommand{\T}{\mathbb{T}}
\newcommand{\N}{\mathbb{N}}
\newcommand{\R}{\mathbb{R}}

\newcommand{\Z}{\mathbb{Z}}

\newcommand{\id}{\hbox{id}}

\renewcommand{\phi}{\varphi}

\modulolinenumbers[5]

\begin{document}

\begin{frontmatter}

\title{Assignability of dichotomy spectrum for discrete time-varying linear control systems}

\author{L.V. Cuong\footnote{Email: cuonglv2@nuce.edu.vn, Department of Information Technology, National University of
Civil Engineering, 55 Giai Phong str., Hanoi, Vietnam}\quad and\quad T.S. Doan\footnote{Email: dtson@math.ac.vn, Institute of Mathematics, Vietnam Academy of Science and Technology, 18 Hoang Quoc Viet, Ha Noi and Thang Long Institute of Mathematics and Applied sciences, Thang Long University, Nghiem Xuan Yem Road, Hoang Mai, Hanoi, Vietnam}}

\begin{abstract}
In this paper, we show that for discrete time-varying linear control systems uniform complete controllability implies arbitrary assignability of dichotomy spectrum of closed-loop systems. This result significantly strengthens the result in \cite{Babiarz} about arbitrary assignability of Lyapunov spectrum of discrete time-varying linear control systems.
\end{abstract}

\begin{keyword}
\texttt{Time-varying systems, Stability, Dichotomy spectrum, Control systems, Lyapunov spectrum}
\MSC[2010] 34A30, 34D09, 34D08, 34H05. 
\end{keyword}

\end{frontmatter}


\section{Introduction}
The notion of dichotomy spectrum of linear time-varying systems initiated from the work of Sacker and Sell in 1970s (see \cite{SackerSell1978}). Since then this notion has played an important role in the qualitative theory of time-varying systems including the stability theory (see \cite{Barreira}), the linearization theory (see \cite{Cuong,Palmer1974}), the invariant manifold theory (see \cite{APS,PS,Barreira}), the normal form theory (see \cite{Siegmund2002b}), the bifurcation theory (see \cite{Rasmussen}), etc....

Due to the wide application of dichotomy spectrum in the qualitative theory of time-varying systems, it is of particular importance to know whether we can control this spectrum. More concretely, we are interested in discrete time-varying linear control system
\[
x_{n+1}=A_nx_n+B_n u_n.
\]
The question is that for a given compact set written as the union of some disjoint intervals  whether there exists a linear feedback $u_n=U_nx_n$ for which the dichotomy spectrum of the closed-loop system
\[
x_{n+1}=(A_n+B_nU_n)x_n
\]
is equal to the given compact set (assignability of dichotomy spectrum). In this paper, we show that uniform complete controllability implies assignability of dichotomy spectrum.

Note that uniform complete controllability is also a sufficient condition for arbitrary assignability of Lyapunov spectrum of time-varying control systems, see  \cite{Popova,Babiarz,Babiarz_2018}. Recall that the  Lyapunov spectrum  of a time-varying system consists of all possible average growth rates of solutions of this system and it is known that the Lyapunov spectrum is a subset of the dichotomy spectrum. Then, our result in assigning dichotomy spectrum implies the result of assigning Lyapunov spectrum in \cite{Babiarz}, see Remark \ref{Comparison} for a more details.

The structure of the paper is follows: The first part of Section \ref{Section2} is devoted to present the basic concept called dichotomy spectrum of discrete time-varying systems (Subsection \ref{Subsection2.1}). The statement of the main result about assignability of dichotomy spectrum is stated in Subsection \ref{Subsection2.2}. The proof of the main result is presented in Subsection \ref{Subsection3.3} of Section \ref{Section3}. The other two subsections of Section \ref{Section3} are preparation for the proof and have the following structure: Subsection \ref{Subsection3.1} is devoted to prove a result on the dichotomy spectrum of upper-triangular discrete time-varying systems, Subsection \ref{Subsection3.2} is used to recall a result in \cite{Babiarz} in  transforming an uniformly completely controllable linear systems to upper-triangular linear systems. In the Appendix, we recall the notion of dichotomy spectrum for continuous time-varying systems. A relation between the dichotomy spectral  of continuous time-varying systems and the associated $1$-time discrete time-varying systems is established  in Lemma \ref{TechnicalLemma}. \\

\noindent
\textbf{Notations}: For $d,s\in\N$, let $\mathcal L^{\infty}(\T,\R^{d\times s})$, where $\T$ stands for $\Z,\Z_{\geq 0}, \Z_{\leq 0}$, denote the space $M=(M_n)_{n\in\T}$ with $M_n\in\R^{d\times s}$ satisfying that
\[
\|M\|_{\infty}:=\sup_{n\in\T} \|M_n\|<\infty.
\]
For $d\in\N$, let $\mathcal L^{\rm Lya}(\T,\R^{d\times d})$ denote the set of all Lyapunov sequences  $M=(M_n)_{n\in\T}$ in $\R^{d\times d}$, i.e. $M\in \mathcal L^{\infty}(\T,\R^{d\times d})$ and its inverse sequence $M^{-1}:=(M_n^{-1})_{n\in\T}$ exists and $M^{-1}\in \mathcal L^{\infty}(\T,\R^{d\times d})$.
\section{Preliminaries and main results}\label{Section2}
\subsection{Dichotomy spectrum of discrete time-varying linear system}\label{Subsection2.1}
Consider discrete time-varying linear system
\begin{equation}\label{ED_01}
x_{n+1}=M_nx_n,\qquad \hbox{for } n\in\Z,
\end{equation}
where $ M:=(M_n)_{n\in\Z}\in \mathcal L^{\rm Lya}(\Z,\R^{d\times d})$. Let $\Phi_M(\cdot,\cdot):\Z\times \Z \rightarrow \R^{d\times d}$ denote the \emph{evolution operator} generated by \eqref{ED_01}, i.e.
\[
\Phi_M(m,n):=
\left\{
\begin{array}{ll}
M_{m}\dots M_{n+1}, & \hbox{ if } m>n,\\[1ex]
\id, & \hbox{ if } m=n,\\[1ex]
M_{m+1}^{-1}\dots M_{n}^{-1}, & \hbox{ if } m<n.
\end{array}
\right.
\]
Next, we introduce the notion of one-sided and two-sided dichotomy spectrum of \eqref{ED_01}. These notions are defined in term of exponential dichotomy. Recall that system \eqref{ED_01} is said to admit an exponential dichotomy on $\T$, where $\T$ is either $\Z,\Z_{\geq 0}$ or $\Z_{\leq 0}$, if there exist $K,\alpha>0$ and a family of projection $(P_n)_{n\in\T}$ in $\R^{d\times d}$ such that for all $m,n\in\T$ we have
\[
\begin{array}{cll}
\|\Phi_M(m,n)P_n\| & \leq K e^{-\alpha(m-n)} & \quad \hbox{ for } m\geq n,\\[1.5ex]
\|\Phi_M(m,n)(\id-P_n)\| & \leq K e^{\alpha(m-n)} & \quad \hbox{ for } m\leq n,
\end{array}
\]
see \cite{Poetyzche}.
\begin{definition}[Dichtomy spectrum for discrete time-varying linear systems]\label{Definition_DiscreteED}
The \emph{dichotomy spectrum} of \eqref{ED_01} on $\Z, \Z_{\geq 0}, \Z_{\leq 0}$ are defined, respectively, as follows
\begin{eqnarray*}
\Sigma_{\rm ED}(M)
&:=&
\big\{\gamma\in\R: x_{n+1}=e^{-\gamma }M_n x_n \hbox{ has no ED on } \Z\big\}, 	\\[1.5ex]
\Sigma_{\rm ED}^{+}(M)
&:=&
\big\{\gamma\in\R: x_{n+1}=e^{-\gamma}M_n x_n \hbox{ has no ED on } \Z_{\geq 0}\big\}, 	\\[1.5ex]
\Sigma_{\rm ED}^{-}(M)
&:=&
\big\{\gamma\in\R: x_{n+1}=e^{-\gamma}M_n x_n \hbox{ has no ED on } \Z_{\leq 0}\big\}. 	
	\end{eqnarray*}
\end{definition}

\begin{remark}
In \cite{Aulbach,Poetyzche}, the definition of dichotomy spectrum is slightly differential to Definition \ref{Definition_DiscreteED} in which the authors consider the shifted systems of the form 
\[
x_{n+1}=\frac{1}{\beta}M_nx_n,\qquad\hbox{where } \beta\in (0,\infty).
\]
Since there is an one-to-one correspondence between $\beta\in (0,\infty)$ and $e^{-\gamma}$, where $\gamma\in \R$, there is an one-to-one correspondence between the spectral in Definition \ref{Definition_DiscreteED} and the ones introduced in \cite{Aulbach,Poetyzche}.
\end{remark}
Thanks to the above Remark and the spectral theorem proved in \cite{Aulbach,Poetyzche}, the spectrum $\Sigma_{\rm ED}(M)$ (also $\Sigma_{\rm ED}^{+}(M)$ and $\Sigma_{\rm ED}^{-}(M)$) is given as the union of at most $d$ disjoint intervals. The corresponding notions of dichotomy spectrum of continuous time-varying linear systems are introduced in the Appendix.
\subsection{Setting and the statement of the main result}\label{Subsection2.2}
Consider a discrete time-varying linear control system
\begin{equation}\label{MainEq}
x_{n+1}=A_nx_n+ B_n u_n,
\end{equation}
where $A=(A_n)_{n\in\mathbb Z}\in \mathcal L^{\rm Lya}(\Z,\R^{d\times d}), B=(B_n)_{n\in\mathbb Z}\in \mathcal L^{\infty}(\Z,\R^{d\times s})$. Let $x(\cdot,n,\xi,u)$ denote the solution of \eqref{MainEq} satisfying that $x(n)=\xi$. Now, we recall the notion of uniform complete controllability of \eqref{MainEq}, see also \cite{Babiarz}.
\begin{definition}[Uniform complete controllability]\label{UniformControllability}
System \eqref{MainEq} is called \emph{uniformly completely controllable} if there exist a positive $\alpha$ and a natural number $K$ such that for all $\xi\in\R^d$ and $k_0\in\Z$ there exists a control sequence $u_n, n=k_0,k_0+1,\dots,k_0+K-1$ such that
\[
x(k_0+K,k_0,0,u)=\xi
\]
and
\[
\|u_n\|\leq \alpha \|\xi\|\qquad\hbox{for all } n=k_0,k_0+1,\dots,k_0+K-1.
\]
\end{definition}
For a bounded sequence of linear feedback control $U=(U_n)_{n\in\Z}\in\mathcal L^{\infty}(\Z,\R^{s\times d})$, the corresponding closed-loop system is
\begin{equation}\label{Closedloop}
x_{n+1}=(A_n+B_n U_n)x_n.
\end{equation}
In the case that $A+BU\in \mathcal L^{\rm Lya}(\Z,\R^{d\times d})$, the dichotomy spectrum of \eqref{Closedloop} is denoted by $\Sigma_{\mathrm{ED}}(A+BU)$.
\begin{definition} The dichotomy spectrum of \eqref{Closedloop} is called \emph{assignable} if for arbitrary disjoint closed intervals $[a_1,b_1],\dots,[a_{\ell},b_{\ell}]$, where $1\leq \ell\leq d$, there exists a bounded linear feedback control $U\in L^{\infty}(\Z,\R^{s\times d})$ such that $A+BU\in \mathcal L^{\rm Lya}(\Z,\R^{d\times d})$ and
	\[
	\Sigma_{\mathrm{ED}}(A+BU)=\bigcup_{i=1}^{\ell}[a_i,b_i].
	\]
\end{definition}
We now state the main result of this paper.
\begin{theorem}[Assignability for dichotomy spectrum of discrete time-varying linear systems]\label{MainTheorem}
Suppose that system \eqref{MainEq} is uniformly completely controllable. Then, the dichotomy spectrum of \eqref{Closedloop} is assignable.
\end{theorem}
\begin{remark}\label{Comparison}
(i) Recall that for a discrete time-varying linear system
\begin{equation}\label{Remark_Eq1}
x_{n+1}=M_nx_n,\qquad \hbox{where } M:=(M_n)_{n\in\Z}\in \mathcal L^{\rm Lya}(\Z,\R^{d\times d}),
\end{equation}
the \emph{Lyapunov exponent} of a non-trivial solution $\Phi_M(n,0)\xi$ of \eqref{Remark_Eq1} is given by
\[
\chi(\xi):=\limsup_{n\to\infty}\frac{1}{n}\log\|\Phi_M(n,0)\xi\|.
\]
The Lyapunov spectrum of \eqref{Remark_Eq1} is defined as
\[
\Sigma_{\rm Lya}(M):=\bigcup_{0\not=\xi\in\R^d} \chi(\xi).
\]
It is known that $\Sigma_{\rm Lya}(M)$ consists of at most $d$ elements (cf. \cite[Chapter II]{Andrianova}). Furthermore, suppose that  $\Sigma_{\rm ED}(M)$ is represented as a disjoint union  of $\ell$ intervals $\bigcup_{i=1}^{\ell}[a_i,b_i]$. Then,
\begin{equation}\label{Remark_Eq2}
\Sigma_{\rm Lya}(M)\subset \Sigma_{\rm ED}(M),\qquad \Sigma_{\rm Lya}(M)\cap [a_i,b_i]\not=\emptyset,
\end{equation}
see, e.g. \cite{Johnson}.

(ii) Suppose that system \eqref{MainEq} is uniformly completely controllable. Now, let $\{\lambda_1,\dots,\lambda_{\ell}\}$ be an arbitrary set of $\ell$ real numbers, where $1\leq \ell\leq d$. Let $a_i=b_i=\lambda_i$ for $1\leq i\leq \ell$. By virtue of Theorem \ref{MainTheorem}, there exists a bounded linear feedback control $U=(U_n)_{n\in\Z}$ such that $A+BU\in \mathcal L^{\rm Lya}(\Z,\R^{d\times d})$ and $\Sigma_{\rm ED}(A+BU)=\bigcup_{i=1}^{\ell}\{\lambda_i\}$. This together with \eqref{Remark_Eq2} implies that
\[
\Sigma_{\rm ED}(A+BU)=\Sigma_{\rm Lya}(A+BU)=\bigcup_{i=1}^{\ell}\{\lambda_i\}.
\]
Consequently, for discrete time-varying linear control systems assignability of dichotomy spectrum implies assignability of Lyapunov spectrum.
\end{remark}
\section{Proof of the main results}\label{Section3}
The main ingredient of the proof consists of two parts. In the first part, we extent a result in \cite{Battelli} to obtain an explicit computation of the dichotomy spectrum of a special upper-triangular linear difference system. Concerning the second part, we first extend the result in \cite[Theorem 4.6]{Babiarz} to two-sided linear systems and then use this result to find a suitable linear feedback control such that the closed-loop system \eqref{Closedloop} is kinematically equivalent to an upper traingular linear difference system.

\subsection{Dichotomy spectrum of upper-triangular linear difference systems}\label{Subsection3.1}
In the first part of this subsection, we extend a part of the result about the presentation of dichotomy spectrum of a block upper-triangular differential equations in terms of the dichotomy spectrum of subsystems in \cite{Battelli} to discrete time-varying systems. To do this, we recall this result for continuous time-varying systems.
\begin{theorem}\label{TechnicalTheorem}
Consider an upper-triangular linear differential equation
\begin{equation*}\label{BlockForm_DifferentialEquations}
\dot x(t)=W(t)x(t),\qquad\hbox{where } W(t)=
\left(
\begin{array}{cc}
X(t) & Z(t)\\[0.5ex]
0 & Y(t)
\end{array}
\right),
\end{equation*}
where $X:\R\rightarrow \R^{k\times k}, Y: \R\rightarrow \R^{(d-k)\times (d-k)}, Z:\R\rightarrow \R^{k\times (d-k)}$ are measurable and essentially bounded. Then,
\begin{equation*}\label{Relation_ContinousTime}
\Sigma^{\pm}_{\rm ED}(X)\cup\Sigma^{\pm}_{\rm ED}(Y)
\subset
\Sigma_{\rm ED}(W)\subset
\Sigma_{\rm ED}(X)\cup\Sigma_{\rm ED}(Y),
\end{equation*}
where $\Sigma^{\pm}_{\rm ED}(X):=\Sigma^{+}_{\rm ED}(X)\cup \Sigma^{-}_{\rm ED}(X), \Sigma^{\pm}_{\rm ED}(Y):=\Sigma^{+}_{\rm ED}(Y)\cup \Sigma^{-}_{\rm ED}(Y)$.
\end{theorem}
\begin{proof}
See \cite[Section 4]{Battelli}.
\end{proof}
Consider discrete time-varying system
\begin{equation}\label{BlockForm}
x_{n+1}=D_nx_n,\qquad\hbox{where } D_n=
\left(
\begin{array}{cc}
A_n & C_n\\[0.5ex]
0 & B_n
\end{array}
\right),
\end{equation}
where $A=(A_n)_{n\in\Z}\in\ \mathcal L^{\rm Lya}(\Z,\R^{k\times k}), B=(B_n)_{n\in\Z}\in\mathcal L^{\rm Lya}(\Z,\R^{(d-k)\times (d-k)})$, and $C=(
C_n)_{n\in\Z}\in\mathcal L^{\infty}(\Z,\R^{k\times (d-k)})$.
\begin{theorem}[Dichotomy spectrum of upper-triangular discrete time-varying linear systems]\label{KeyTheorem}
Let $\Sigma_{\rm ED}(D)$ denote the dichotomy spectrum of \eqref{BlockForm}. Then,
\begin{equation}\label{Relation}
\Sigma^{\pm}_{\rm ED}(A)\cup\Sigma^{\pm}_{\rm ED}(B)
\subset
\Sigma_{\rm ED}(D)\subset
 \Sigma_{\rm ED}(A)\cup\Sigma_{\rm ED}(B),
\end{equation}
where $\Sigma^{\pm}_{\rm ED}(A):=\Sigma^{+}_{\rm ED}(A)\cup \Sigma^{-}_{\rm ED}(A), \Sigma^{\pm}_{\rm ED}(B):=\Sigma^{+}_{\rm ED}(B)\cup \Sigma^{-}_{\rm ED}(B)$.
\end{theorem}
\begin{proof}
Define a measurable and bounded function $W:\R\rightarrow \R^{d\times d}$ of the form $W(t)=\left(
\begin{array}{cc}
X(t) & Z(t)\\[0.5ex]
0 & Y(t)
\end{array}
\right)$, where
\[
X(t)= A_n, Y(t)=B_n, Z(t)=C_n\qquad \hbox{ for } t\in [n,n+1), n\in\Z.
\]
Obviously, equation \eqref{BlockForm} is the $1$-time discrete time-varying systems associated with
\begin{equation*}\label{Cont_Eq1}
\dot x=W(t)x,\qquad \hbox{ where } t\in\R,
\end{equation*}
(see Appendix for the notion of the associated $1$-time discrete time-varying systems). Then, by virtue of Lemma \ref{TechnicalLemma} we have
\begin{equation}\label{New_Eq1}
\begin{array}{ccc}
\Sigma^{\pm}_{\rm ED}(A)\cup\Sigma^{\pm}_{\rm ED}(B)&=& \Sigma^{\pm}_{\rm ED}(X)\cup\Sigma^{\pm}_{\rm ED}(Y),\\[1ex]
\Sigma_{\rm ED}(D)&=& \Sigma_{\rm ED}(W),\\[1ex]
\Sigma_{\rm ED}(A)\cup\Sigma_{\rm ED}(B)&=&\Sigma_{\rm ED}(X)\cup\Sigma_{\rm ED}(Y).
\end{array}
\end{equation}
On the other hand, by definition of $W(t)$ and Theorem \ref{TechnicalTheorem} we have
\[
\Sigma^{\pm}_{\rm ED}(X)\cup\Sigma^{\pm}_{\rm ED}(Y)
\subset
\Sigma_{\rm ED}(W)\subset
 \Sigma_{\rm ED}(X)\cup\Sigma_{\rm ED}(Y),
\]
which together with \eqref{New_Eq1} proves \eqref{Relation}. The proof is complete.
\end{proof}
In the final part of this subsection, we study a special class of upper triangular discrete time-varying systems whose dichotomy spectrum are given as the union of the dichotomy spectrum of the subsystems corresponding to diagonal entries. More concretely,  let $(p^1_n)_{n\in\Z},(p^2_n)_{n\in\Z},\dots,(p^d_n)_{n\in\Z}$ be scalar Lyapunov sequences satisfying that
\begin{equation}\label{Eq1}
p^i_{n}=p^i_{-n}\qquad\hbox{ for all } n\in\Z, i=1,\dots,d.
\end{equation}
For each $i=1,\dots,d$, we denote by $\Sigma_{\rm ED}(p^i)$ the dichotomy spectrum of the scalar linear system
\[
z_{n+1}= p^i_n z_n\qquad\hbox{ for } n\in\Z.
\]
\begin{proposition}\label{MainProposition}
Let $(C_n)_{n\in\Z}$, where $C_n=(c^{(n)}_{ij})_{1\leq i,j\leq d}$, be an arbitrary bounded sequence of upper-triangular matrices in $\R^{d\times d}$ satisfying that
\[
c^{(n)}_{ii}= p^i_n\qquad \hbox{ for all } n\in \Z, i=1,\dots,d.
\]
Then, the dichotomy spectrum $\Sigma_{\rm ED}(C)$ of the system $x_{n+1}=C_n x_n$ is given by
\[
\Sigma_{\rm ED}(C)=\bigcup_{i=1}^d \Sigma_{\rm ED}(p^i).
\]
\end{proposition}
\begin{proof}
Using Theorem \ref{KeyTheorem}, we obtain that
\[
\bigcup_{i=1}^d \Sigma_{\rm ED}^{\pm}(p^i)\subset\Sigma_{\rm ED}(C)\subset \bigcup_{i=1}^d \Sigma_{\rm ED}(p^i),
\]
where $\Sigma_{\rm ED}^{\pm}(p^i)=\Sigma_{\rm ED}^{+}(p^i)\cup \Sigma_{\rm ED}^{-}(p^i)$. Thus, to complete the proof it is sufficient to show that
\begin{equation}\label{Aim}
\Sigma_{\rm ED}(p^i)\subset \Sigma_{\rm ED}^{\pm}(p^i)\qquad\hbox{ for all } i=1,\dots,d.
\end{equation}
For this purpose, let $i\in\{1,\dots,d\}$ and $\gamma\not\in \Sigma_{\rm ED}^{+}(p^i)$ be arbitrary. Then, by Definition \ref{Definition_DiscreteED} one of the following alternatives holds:\\

\noindent
\emph{(A1)} There exist $K,\alpha>0$ such that
\begin{equation}\label{A1}
|p^i_{m-1}\dots p^i_n|\leq K e^{(\gamma-\alpha)(m-n)}
\qquad\hbox{ for } m,n\in\Z_{\geq 0}\hbox{ with } m\geq n.
\end{equation}
Thus, by \eqref{Eq1} we also have that
\[
|p^i_{m-1}\dots p^i_n|=
\left\{
\begin{array}{ll}
 |p^i_{m-1}\dots p^i_0||p^i_{1}\dots p^i_{-n}|\leq K^2 e^{(\gamma-\alpha)(m-n)} & \hbox{ for } m\geq 0\geq n,\\[1.5ex]
|p^i_{-(m-1)}\dots p^i_{-n}|\leq K e^{(\gamma-\alpha)(m-n)} & \hbox{ for } 0\geq m \geq n.
\end{array}
\right.
\]
It means that the shifted system
\[
z_{n+1}=e^{-\gamma}p^i_nz_n,\qquad\hbox{where } n\in\Z
\]
exhibits an exponential dichotomy on $\Z$. Consequently, $\gamma\not\in \Sigma_{\rm ED}(p^i)$.\\

\noindent
\emph{(A2)} There exist $K,\alpha>0$ such that
\begin{equation*}
\left|\frac{1}{p^i_{m}}\dots \frac{1}{p^i_{n-1}}\right|\leq K e^{(\gamma+\alpha)(m-n)}
\qquad\hbox{ for } m,n\in\Z_{\geq 0}\hbox{ with } m\leq n,
\end{equation*}
which implies that
\begin{equation*}
\left|p^i_{m}\dots p^i_{n-1}\right|\geq \frac{1}{K} e^{(\gamma+\alpha)(n-m)}
\qquad\hbox{ for } m,n\in\Z_{\geq 0}\hbox{ with } n\geq m,
\end{equation*}
Thus, by \eqref{Eq1} we also have that
\[
\left|p^i_{m}\dots p^i_{n-1}\right|
=
\left\{
\begin{array}{ll}
 |p^i_{-m}\dots p^i_{-1}||p^i_{0}\dots p^i_{n-1}|\geq \frac{1}{K^2} e^{\gamma(m-n)} & \hbox{ for } n\geq 0\geq m,\\[1.5ex]
|p^i_{-m}\dots p^i_{-(n-1)}|\geq \frac{1}{K} e^{\gamma(m-n)} & \hbox{ for } 0\geq n\geq m.
\end{array}
\right.
\]
It means that the shifted system
\[
z_{n+1}=e^{-\gamma}p^i_nz_n,\qquad\hbox{where } n\in\Z
\]
exhibits an exponential dichotomy on $\Z$. Therefore, in this alternative we also arrive at $\gamma\not\in \Sigma_{\rm ED}(p^i)$.

Since $\gamma\not\in  \Sigma_{\rm ED}^{+}(p^i)$ is arbitrary it follows that $\Sigma_{\rm ED}(p^i)\subset \Sigma_{\rm ED}^{+}(p^i)$. This shows \eqref{Aim} and the proof is complete.
\end{proof}

\subsection{Upper-triangularization of uniformly completely controllable systems}\label{Subsection3.2}
Recall that two discrete time-varying linear systems
\[
x_{n+1}=A_n x_n,\quad y_{n+1}=B_n y_n\qquad \hbox{ for } n\in\T \quad (\T~ \hbox{stands for} ~\Z_{\ge 0}~\hbox{or}~ \Z),
\]
where $(A_n)_{n\in\T}, (B_n)_{n\in\T}\in \mathcal L^{\rm Lya}(\T,\R^{d\times d})$, are called \emph{kinematically equivalent} (or also called dynamically equivalent) if there exists a transformation $(T_n)_{n\in\T}\in \mathcal L^{\rm Lya}(\T,\R^{d\times d})$ such that
\[
A_nT_n=T_{n+1}B_n\qquad\hbox{ for all } n\in\T.
\]
As was proved in \cite[Theorem 4.6]{Babiarz} that for an uniformly completely controllable one sided discrete time-varying control system and a given diagonal discrete time-varying system, there is a bounded feedback control such that the corresponding closed-loop system is dynamically equivalent to an upper-triangular system whose diagonal part coincides with the given diagonal system. Under a slight modification, this result can be extended to two-sided discrete time-varying control system and we arrive at the following result.
\begin{theorem}[Upper-triangularization of uniformly completely controllable two sided discrete time-varying systems]\label{Upper_Theorem} Consider an uniformly completely controllable two-sided discrete time-varying control system
\begin{equation}\label{Upper_Eq1}
x_{n+1}=A_nx_n+B_n u_n,\qquad\hbox{for } n\in\Z,
\end{equation}
where $ A=(A_n)_{n\in\mathbb Z}\in \mathcal L^{\rm Lya}(\Z,\R^{d\times d}), B=(B_n)_{n\in\mathbb Z}\in \mathcal L^{\infty}(\Z,\R^{d\times s})$. Let $(p^i_n)_{n\in\Z}, i=1,\dots,d,$ be arbitrary scalar positive Lyapunov sequences. Then, there exist a sequence of upper triangular matrices $(C_n)_{n\in\Z}\in \mathcal L^{\rm Lya}(\Z,\R^{d\times d})$, where $C_n=(c^{(n)}_{ij})_{1\leq i,j\leq d}$ with $c^{(n)}_{ii}=p^i_n$, and a bounded feedback control $U=(U_n)_{n\in\N}\in \mathcal L^{\infty}(\Z,\R^{s\times d})$ satisfying that the following systems
\begin{equation*}
x_{n+1}=(A_n+B_nU_n)x_n,\quad y_{n+1}=C_ny_n
\qquad \hbox{ for } n\in \Z
\end{equation*}
are kinematically equivalent.
\end{theorem}
\begin{proof}
See \cite[Theorem 4.6]{Babiarz}.
\end{proof}
\subsection{Proof of the main result}\label{Subsection3.3}
\begin{proof}[Proof of Theorem \ref{MainTheorem}]
Let $[a_1,b_1],\dots,[a_{\ell},b_{\ell}]$, where $1\leq \ell\leq d$, be arbitrary disjoint closed intervals. For $1\leq i\leq \ell$, we define a positive scalar sequence $(p^i_n)_{n\in\Z}$ with $p^i_n=p^i_{-n}$ for $n\in\Z$ and
\begin{equation}\label{Contruction}
p^i_n=
\left\{
  \begin{array}{ll}
    e^{a_i}, & \hbox{for } n\in [2^{2m},2^{2m+1}), m\in\Z_{\geq 0} ; \\[1ex]
    e^{b_i}, & \hbox{for } n\in [2^{2m+1},2^{2m+2}), m\in\Z_{\geq 0};\\[1ex]
    0, & \hbox{for } n=0.
  \end{array}
\right.
\end{equation}
Consider the corresponding linear scalar system
\begin{equation}\label{CorrespondingSystem}
z_{n+1}=p^i_n z_n\qquad\hbox{ for } n\in\Z.
\end{equation}
By virtue of Proposition \ref{MainProposition}, the dichotomy spectrum of \eqref{CorrespondingSystem} satisfies that   $\Sigma_{\rm ED}(p^i)=\Sigma_{\rm ED}^{+}(p^i)$. By \eqref{Contruction}, it is obvious to see that $\Sigma_{\rm ED}^{+}(p^i)=[a_i,b_i]$ and then we arrive at
\begin{equation}\label{ComputationED}
\Sigma_{\rm ED}(p^i)=[a_i,b_i]\qquad\hbox{ for } i=1,\dots,\ell.
\end{equation}
For $\ell+1\leq i\leq d$, let $p^i_n=p^1_n$. According to Theorem \ref{Upper_Theorem}, there exists a bounded feedback control and a sequence of upper triangular matrices $(C_n)_{n\in\Z}\in \mathcal L^{\rm Lya}(\Z,\R^{d\times d})$, where $C_n=(c^{(n)}_{ij})_{1\leq i,j\leq d}$ with $c^{(n)}_{ii}=p^i_n$ such that
\begin{equation*}
x_{n+1}=(A_n+B_nU_n)x_n,\quad y_{n+1}=C_ny_n
\qquad \hbox{ for } n\in \Z
\end{equation*}
are kinematically equivalent. This together with Proposition \ref{MainProposition} and \eqref{ComputationED} implies that
\[
\Sigma_{\rm ED}(A+BU)=\Sigma_{\rm ED}(C)=\bigcup_{i=1}^d \Sigma_{\rm ED}(p^i)=\bigcup_{i=1}^{\ell}[a_i,b_i].
\]
The proof is complete.
\end{proof}
\section{Appendix}\label{Section4}
Consider a continuous time-varying linear system
\begin{equation}\label{ED_01_Cont}
\dot x(t)=W(t)x(t),\qquad t\in\R,
\end{equation}
where $W:\R\rightarrow \R^{d\times d}$ is measurable and bounded.
Let $\Phi_W(\cdot,\cdot):\R\times \R \rightarrow \R^{d\times d}$ denote the \emph{evolution operator} generated by \eqref{ED_01_Cont}, i.e. $\Phi(\cdot,s)\xi$ solves \eqref{ED_01_Cont} with the initial valued condition $x(s)=\xi$.  Next, we introduce the notion of one-sided and two-sided dichotomy spectrum of \eqref{ED_01_Cont}. These notions are defined in term of exponential dichotomy. Recall that system \eqref{ED_01} is said to admit an exponential dichotomy on $\T$, where $\T$ is either $\R,\R_{\geq 0}$ or $\R_{\leq 0}$, if there exist $K,\alpha>0$ and a family of projection $P:\T\rightarrow \R^{d\times d}$ such that for all $t,s\in\T$ we have
\[
\begin{array}{cll}
\|\Phi_W(t,s)P(s)\| & \leq K e^{-\alpha(t-s)} & \quad \hbox{ for } t\geq s;\\[1.5ex]
\|\Phi_W(t,s)(\id-P(s))\| & \leq K e^{\alpha(t-s)} & \quad \hbox{ for } t\leq s.
\end{array}
\]
\begin{definition}[Dichotomy spectrum for continuous-time varying linear systems]\label{DefinitionofED}
	The dichotomy spectrum of \eqref{ED_01} on $\R, \R_{\geq 0}, \R_{\leq 0}$ are defined, respectively, as follows
	\begin{eqnarray*}
		\Sigma_{\rm ED}(W)
		&:=&
		\big\{\gamma\in\R: \dot x=(W(t)-\gamma \id) x \hbox{ has no ED on } \R\big\}, 	\\[1.5ex]
		\Sigma_{\rm ED}^{+}(W)
		&:=&
		\big\{\gamma\in\R:  \dot x=(W(t)-\gamma \id) x \hbox{ has no ED on } \R_{\geq 0}\big\}, 	\\[1.5ex]
		\Sigma_{\rm ED}^{-}(W)
		&:=&
		\big\{\gamma\in\R:  \dot x=(W(t)-\gamma \id) x \hbox{ has no ED on } \R_{\leq 0}\big\}. 	
	\end{eqnarray*} 	
\end{definition}
It is proved in \cite{Siegmund,Kloeden} that $\Sigma_{\rm ED}(W)$ (also $\Sigma_{\rm ED}^{+}(W)$ and $\Sigma_{\rm ED}^{-}(W)$) is a compact set consisting of at most $d$ disjoint intervals.

Now, we introduce systems associated with \eqref{ED_01_Cont}. The following system
\begin{equation}\label{AssociatedEq}
x_{n+1}=A_n x_n, \qquad \hbox{ where } A_n:=\Phi_W(n+1,n),
\end{equation}
is called the \emph{$1$-time discrete time-varying linear system associated with \eqref{ED_01_Cont}}, see also \cite{Cuong}. Obviously, the evolution operator $\Phi_A(\cdot,\cdot):\Z\times\Z\rightarrow \R^{d\times d}$ is given by
\begin{equation}\label{AssociatedEvolution}
\Phi_A(m,n)=\Phi_W(m,n)\qquad\hbox{for } m,n\in\Z.
\end{equation}
The following lemma shows that the dichotomy spectral of \eqref{ED_01_Cont} and \eqref{AssociatedEq} coincide.
\begin{lemma}\label{TechnicalLemma}
The following statements hold
\[
\Sigma_{\rm ED}(W)=\Sigma_{\rm ED}(A), \Sigma^+_{\rm ED}(W)=\Sigma^+_{\rm ED}(A),\Sigma^{-}_{\rm ED}(W)=\Sigma^{-}_{\rm ED}(A).
\]
\end{lemma}
\begin{proof}
We only prove $\Sigma_{\rm ED}(W)=\Sigma_{\rm ED}(A)$ and by using similar arguments we also have $ \Sigma^+_{\rm ED}(W)=\Sigma^+_{\rm ED}(A),\Sigma^{-}_{\rm ED}(W)=\Sigma^{-}_{\rm ED}(A)$. We divide the proof of this fact into two steps:\\

\noindent \emph{Step 1}: We show that $\Sigma_{\rm ED}(A)\subset \Sigma_{\rm ED}(W)$. For this purpose, let $\gamma\not\in \Sigma_{\rm ED}(W)$ be arbitrary. Then, by Definition \ref{DefinitionofED} and the fact that $e^{-\gamma(t-s)}\Phi_W(t,s)$ is the evolution operator of the shifted systems
\[
\dot x=(W(t)-\gamma \id) x,
\]
there exist $K,\alpha>0$ and a family of projection $P:\R\rightarrow \R^{d\times d}$ such that
\[
\begin{array}{cll}
\|\Phi_W(t,s)P(s)\| & \leq K e^{(\gamma-\alpha)(t-s)} & \quad \hbox{ for } t\geq s,\\[1.5ex]
\|\Phi_W(t,s)(\id-P(s))\| & \leq K e^{(\gamma+\alpha)(t-s)} & \quad \hbox{ for } t\leq s.
\end{array}
\]
In particular, by letting $P_n:=P(n)$ for $n\in\Z$ and \eqref{AssociatedEvolution} we arrive at the following properties of the evolution $\Phi_A(m,n)$ generated by \eqref{AssociatedEq}
\[
\begin{array}{cll}
\|\Phi_A(m,n)P_n\| & \leq K e^{(\gamma-\alpha)(m-n)} & \quad \hbox{ for } m\geq n,\\[1.5ex]
\|\Phi_A(m,n)(\id-P_n\| & \leq K e^{(\gamma+\alpha)(m-n)} & \quad \hbox{ for } m\leq n.
\end{array}
\]
Consequently, the shifted discrete time-varying system
\[
x_{n+1}=e^{-\gamma}A_nx_n,\qquad n\in\Z,
\]
exhibits an exponential dichotomy. Thus, $\gamma\not\in \Sigma_{\rm ED}(A)$.\\

\noindent \emph{Step 2}: We show that $\Sigma_{\rm ED}(W)\subset \Sigma_{\rm ED}(A)$. For this purpose, let $\gamma\not\in \Sigma_{\rm ED}(A)$ be arbitrary. By Definition \ref{DefinitionofED}, there exist $K,\alpha>0$ and a family of projection $(P_n)_{n\in\Z}$ of $\R^{d\times d}$ such that
\begin{equation}\label{Estimate1}
\begin{array}{cll}
\|\Phi_A(m,n)P_n\| & \leq K e^{(\gamma-\alpha)(m-n)} & \quad \hbox{ for } m\geq n,\\[1.5ex]
\|\Phi_A(m,n)(\id-P_n\| & \leq K e^{(\gamma+\alpha)(m-n)} & \quad \hbox{ for } m\leq n.
\end{array}
\end{equation}
We define a map $P:\R\rightarrow \R^{d\times d}$ by
\[
P(t):=\Phi_{W}(t,n)P_n\Phi_W(n,t)\qquad \hbox{ for } t \in [n,n+1), n\in\Z.
\]
Since $W(\cdot)$ is measurable and essentially bounded, i.e. $\mbox{ess}\sup_{t\in\R}\|W(t)\|<\infty$, it follows with Gronwall's inequality that
\[
\kappa:=\sup_{|t-s|\leq 1} \|\Phi_W(t,s)\|<\infty.
\]
Thus, for any $t\geq s$ by letting $m:=\left\lceil t \right\rceil$ (the smallest integer number greater or  equal $t$), $ n:=\left\lfloor s \right\rfloor$ (the largest integer number smaller or equal $s$) and \eqref{Estimate1} we have
\begin{eqnarray*}
\|\Phi_W(t,s)P(s)\|
&=&
\|\Phi_W(t,s)\Phi_W(s,n)P_n\Phi_W(n,s)\|\\[0.5ex]
&\leq &
\kappa^2 \|\Phi_W(m,n)P_n\|\\[0.5ex]
&\leq&
\kappa^2 K e^{2|\gamma-\alpha|} e^{(\gamma-\alpha)(t-s)}.
\end{eqnarray*}
Similarly, for $t\leq s$ we have
\[
\|\Phi_W(t,s)(\id-P(s))\|
\leq
\kappa^2 K e^{2|\gamma+\alpha|} e^{(\gamma+\alpha)(t-s)},
\]
which implies that the shifted continuous time-varying system
\[
\dot x= (W(t)-\gamma \id) x
\]
exhibits an exponential dichotomy. Thus, $\gamma\not\in \Sigma_{\rm ED}(W)$ and the proof is complete.
\end{proof}
\section*{Acknowledgement}
This research is funded by Vietnam National University of Civil Engineering (NUCE) under grant number 202-2018/KHXD-TD. Our attention to the problem of assignability for dichotomy spectrum of linear discrete time-varying systems initiated from the visit of Dr. Artur Babiarz, Prof. Adam Czornik and Dr. Michal Niezabitowski to Institute of Mathematics, VAST in 2017. The authors thank them for an interesting  and highly motivated introduction of this problem.

\end{document}